\documentclass[11pt]{amsart} 
\usepackage[utf8]{inputenc}
\usepackage[T1]{fontenc}
\usepackage{lmodern}
\usepackage{amsmath, amsthm, amssymb, amsfonts}
\usepackage{latexsym}
\usepackage{amsxtra} 
\usepackage[all]{xy}
\usepackage{nicefrac,mathtools}
\usepackage[shortlabels, inline]{enumitem}
\usepackage{microtype}
\usepackage{hyperref}
\usepackage{graphicx,calc}
\newlength\myheight
\newlength\mydepth
\settototalheight\myheight{Xygp}
\usepackage[all]{xy}
\usepackage{xcolor}
\usepackage{tikz-cd}
\usepackage{tikz}
\usepackage[lite]{amsrefs}
\usepackage[T1]{fontenc}
\usepackage[utf8]{inputenc}
\usepackage{thmtools}
\usepackage{geometry}
\numberwithin{equation}{section}
\theoremstyle{plain}
\newtheorem{theorem}[equation]{Theorem}

\newtheorem{Lemma}[equation]{Lemma}

\newtheorem{proposition}[equation]{Proposition}

\newtheorem{corollary}[equation]{Corollary}

\theoremstyle{definition}
\newtheorem{definition}[equation]{Definition}

\theoremstyle{remark} \newtheorem{remark}[equation]{Remark}
\theoremstyle{remark} 
 \newtheorem*{remark*}{Remark}
\newtheorem*{remarks*}{Remark} \newtheorem{example}[equation]{Example}

\newcommand{\Cost}{\text{Cost}}
\newcommand{\R}{\mathbb{R}}
\newcommand{\Z}{\mathbb{Z}}

\DeclarePairedDelimiter{\abs}{\lvert}{\rvert}

\title[Cost Function]{On The Cost Function associated with Legendrian Knots}
\author{Dheeraj Kulkarni}\email{dheeraj@iiserb.ac.in}\address{Department of Mathematics, Indian Institute of Science Education and Research, Bhopal, India}
\author{Tanushree Shah} \email{tanushrees@cmi.ac.in}\address{Chennai Mathematical Institute, India}
\author{Monika Yadav} \email{monika18@iiserb.ac.in}\address{Department of Mathematics, Indian Institute of Science Education and Research, Bhopal, India}

 \thanks{\emph{Subjclass[2020]}: 57K10, 57K14, 57K33 }

\begin{document}
\begin{abstract} In this article, we introduce a non-negative integer-valued function that measures 
the obstruction for converting topological isotopy between two Legendrian knots into a Legendrian 
isotopy. We refer to this function as the Cost function. We show that the Cost function induces a metric on the set of topologically isotopic Legendrian knots. Hence, the set of topologically isotopic 
Legendrian knots can be seen as a graph with path-metric given by the Cost function. Legendrian simple 
knot types are shown to be characterized using the Cost function. We also get a quantitative version of Fuchs-Tabachnikov's Theorem that says any two Legendrian knots in $(\mathbb{S}^3,\xi_{std})$ in the same topological knot type become Legendrian isotopic after sufficiently many stabilizations \cite{FT}. 
We compute the Cost function for Legendrian simple knots (for example torus knots) and we note the behavior of Cost function for twist knots and cables of torus knots (some of which are Legendrian non-simple). We also construct examples of Legendrian representatives of 2-bridge knots and compute the Cost between them. Further, we investigate the behavior of the Cost function under the 
connect sum operation. We conclude with some questions about the Cost function, its relation with the standard contact structure, and the topological knot type.
\end{abstract}

\maketitle
\tableofcontents

\section{Introduction}

A \emph{Legendrian} knot in $\R^3 $ with the standard contact structure $\xi_{st} = \text{ker}(dz-ydx) $ is a smooth
embedding of $\mathbb{S}^1 $ in $\R^3 $ such that its image is tangent to the contact plane at each point. Two
Legendrian knots are called \emph{Legendrian} isotopic if there is a smooth 1-parameter family of Legendrian knots starting with one and ending with the other. 
The classification of Legendrian knots up to Legendrian isotopy is not completely understood. It is easy to see that given a topological embedding of $\mathbb{S}^1$ in $\R^3 $, denoted by $K$, it can be approximated by a Legendrian knot in an arbitrarily small $C^0 $-neighborhood of $K$.
Hence, there are infinitely many Legendrian knots approximating $K$. All of them are topologically isotopic to $K$. We refer to a Legendrian knot that approximates $ K$ as a \emph{Legendrian representative} of the knot type of $K$.

The properties of the contact structure $\xi_{st} $ reflect deeply in the classification problem of Legendrian representatives of a given knot. Sometimes, the classical invariants given by the Thurston-Bennequin number ($tb$) and the rotation number ($rot$) completely classify Legendrian representatives up to Legendrian isotopy.
On the other hand, there are knot types where classical invariants are not sufficient to tell the Legendrian representatives apart (see \cite{C1},\cite{C2}). Deep mathematical ideas based on counting J-holomorphic curves and packaging the information into differential graded algebras (DGAs) were developed by Eliashberg-Chekanov (\cite{EF}) and Ng (\cite{LNg}) to classify Legendrian representatives. Additionally, Legendrian knots are classified by invariants such as $\rho$-rulings. For exposition on $\rho$-rulings refer to \cite{E}.

In contrast, there have not been many attempts to understand the \emph{space} of all Legendrian representatives of a given knot type. Whether the set of all Legendrian representatives of a given knot type admits a meaningful and natural structure is the question in focus. Mountain ranges formed by Legendrian representatives based on $tb $ and $rot $ have been useful in a limited manner (\cite{E}). The limitation of mountain ranges arises from the fact that $tb$ and $rot$ are not sufficient for many families of knot types.

In this article, we introduce a function that assigns a non-negative integer to a pair of Legendrian representatives of a given knot type. Roughly speaking, this function measures the obstruction to converting a topological isotopy between two Legendrian representatives into a Legendrian isotopy. The output of this function, an integer, tells us the minimum cost we have to pay to upgrade a topological isotopy into a Legendrian isotopy. Hence we refer to it as the `Cost function'. We show that the Cost function gives a metric (see Section \ref{Costgraph}) on the set of Legendrian representatives of a given knot type turning it into a metric space. Therefore, we get a graph of Legendrian representatives where path metric is given by the Cost function.

We investigate the properties of the Cost function. We give bounds on the Cost function in terms of $tb$
and $rot$ (Lemma \ref{CostL9}). We obtain a characterization of Legendrian simple knot types in terms of an explicit formula for the Cost function. We give a quantitative version of Fuchs-Tabachnikov's Theorem (\cite{FT}). We study the effect of Legendrian connect-sum operation on the Cost function (\cite{C}). We give examples of Legendrian twist knots and torus knots together with values of Cost function. 
Further, we compute the Cost function of certain new families of 2-bridge knots which are distinguished using $\rho$-rulings (see \ref{2-b}). 
Finally, we raise a few questions for further understanding of the behavior of the Cost function from theoretical as well as computational viewpoints. \\

\noindent {\bf Acknowledgements:} The second author was partly supported by "Singularities and Low Dimensional Topology" semester at Alfr\'ed R\'enyi Institute, Hungary, and Chennai Mathematical Institute, India.

The third author was supported by the CSIR grant 09/1020(0152)/2019-EMR-I, DST,
Government of India.
\section{Preliminaries} 
Throughout this paper, we will be working with the contact manifold $(\mathbb{R}^3,\xi_{st})$, where $\xi_{st}:=ker(dz-ydx)$. For a detailed exposition, we refer the reader to \cite{E}.\ We briefly recall the notions for the convenience of the reader.

\ 

A Legendrian knot in $(\mathbb{R}^3,\xi_{st})$ is a smooth embedding $\gamma:\mathbb{S}^1\rightarrow \mathbb{R}^3$ such that $\gamma'(t)\in \xi_{st}(\gamma(t)) \:\forall t \in \mathbb{S}^1.$ Such an embedding is called a Legendrian embedding. A front projection of a Legendrian knot in $(\mathbb{R}^3,\xi_{st})$ is its projection onto the $xz$-plane.

\begin{figure}[!htbp]
        \centering
        \includegraphics[scale=0.7]{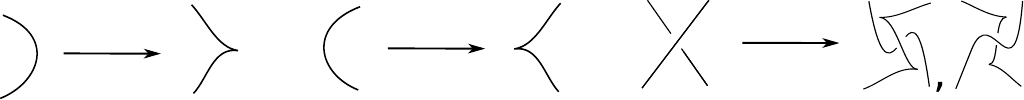}
        \caption{Converting a knot diagram into a front projection.}
        \label{fig:CTF}
    \end{figure}

    A knot diagram $D$ can be converted into a front projection of a Legendrian knot in the following way:  replace all the vertical tangencies with cusps and whenever at a crossing, the slope of overcrossing is more positive than the slope of undercrossing modify it by adding two cusps as shown in Figure \ref{fig:CTF}.

We have adopted a version of the Reidemeister theorem for Legendrian knots from \cite{SJ}. 
    \begin{theorem}[\cite{SJ}]
        Let $K_1$ and $K_2$ be Legendrian knots. Then $K_1$ is Legendrian isotopic to $K_2$ if and only if their front projections are related by regular homotopy and a sequence of moves shown in Figure \ref{fig:LRmoves}.
     \end{theorem}
     \begin{figure}[!htbp]
    \centering
    \includegraphics[scale=0.8]{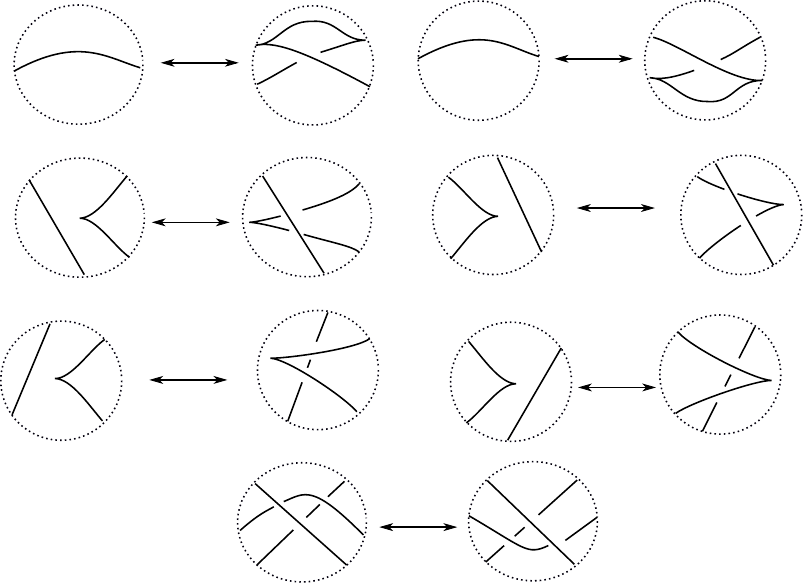}
    \caption{The moves shown in this figure are called Legendrian Reidemeister (LR) moves. We will denote each move in the first row by LR1, each move in the second and third rows by LR2, and the move in the last row by LR3.}
    \label{fig:LRmoves}
\end{figure}

Let $K$ be a Legendrian knot with a front projection $F$.
An oriented arc in a front projection $F$ can be replaced by an arc with two downwards (upwards) cusps as shown in Figure \ref{fig:stabilization}. This process is called a \textit{positive (negative) stabilization} of $K$. We use the notation $S^{\pm}(K)$ and $S^{\pm}(F)$ to denote a positive (negative) stabilization of $K$ and $F$ respectively.
We also use the notation $(S^{\pm})^p$ to denote $p$-many positive (negative) stabilizations.  A stabilization of $K$ is independent of its location in the front projection up to Legendrian isotopy. Adding a stabilization changes $tb$ by $-1$ and rotation number by $\pm 1$. 
 \begin{figure}
        \centering
        \includegraphics[scale=0.5]{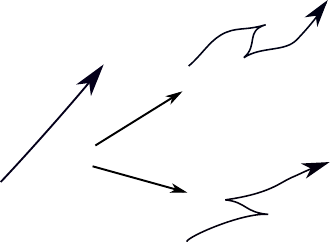}
        \caption{Positive and a negative stabilization}
        \label{fig:stabilization}
    \end{figure}
    
A (\textit{Legendrian}) \textit{topological knot type} is an equivalence class of (Legendrian) topological knots up to (Legendrian) topological isotopy. For a topological knot type $\mathcal{K}$, we use the notation $\mathcal{L(K)}$ to denote the set of all Legendrian representatives of $\mathcal{K}$.

We now give a brief introduction to the connect sum operation on Legendrian knots in $\mathbb{S}^3$ with the standard tight contact structure. A detailed description of the connected sum of Legendrian knots is given in \cite{EH}. In this article, 
we will be only working with the diagrammatic interpretation of the connected sum using the front projections as shown in Figure \ref{fig:LCS}. The front projection $F_1\#F_2$ is obtained by resolving a right cusp from $F_1$ and a 
left cusp from $F_2$. One can refer to Section 4 of \cite{EH} to see that the connected sum operation shown in Figure \ref{fig:LCS} is independent of the choices of the left cusp and the right cusp.

\begin{figure}[!htbp]
    \centering
    \includegraphics[scale=0.7]{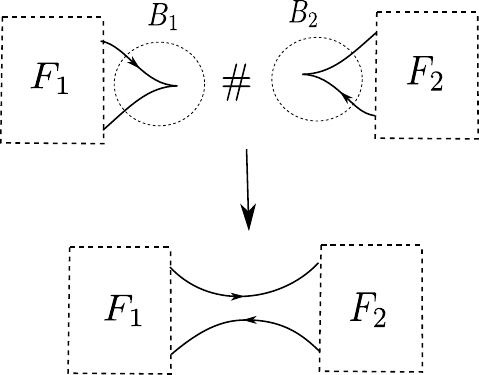}
    \caption{Connected sum of $F_1$ and $F_2$.}
    \label{fig:LCS}
\end{figure}

\begin{remark}\label{newRemark1}
    It follows from a simple observation that if $K_1\#K_2$ is null-homologous then $tb(K_1\#K_2)=tb(K_1)+tb(K_2)+1$ and $rot(K_1\#K_2)=rot(K_1)+rot(K_2)$, see \cite{EH}.
\end{remark}

\begin{remark}\label{newRemark3}
    Using Remark \ref{newRemark1} and Corollary 3.5, in \cite{EH}, the connected sum of maximal $tb$ Legendrian knots is a maximal $tb$ Legendrian knot.
\end{remark}
\begin{remark}\label{newRemark}
    Legendrian prime decomposition of a Legendrian knot is unique up to moving stabilizations from one component to another and possible permutations of components (\cite{EH}). 
\end{remark}

\begin{remark}
    Section 4 onwards, all the knots are oriented and a (Legendrian) topological knot type is considered up to the (Legendrian) isotopy of oriented (Legendrian) topological knots.
\end{remark}

Now onwards, by `R-moves', we mean smooth Reidemeister moves, and by `LR-moves' we mean Legendrian Reidemeister moves.

\section{Cost Function}

Fuchs-Tabachnikov \cite{FT} proved that any two Legendrian knots, that are topologically isotopic, can be stabilized sufficiently many times to obtain 
Legendrian isotopic knots. In the light of this theorem, we define the Cost function which measures the obstruction for the two Legendrian knots in the same topological type from becoming Legendrian isotopic. 
Our situation reduces to having two front diagrams that are topologically isotopic. Assume that we have front diagrams $F_0$ and $F_1$ that represent topologically isotopic knots. Then there is a sequence of Reidemeister moves and 
planar isotopies converting the knot diagram $F_0$ to $F_1$. 

    Our goal is to find out how one can upgrade a topological isotopy to a Legendrian isotopy. 
Hence the goal is twofold: convert Reidemeister moves into Legendrian Reidemeister moves and convert planar isotopies into planar isotopies that lift to Legendrian isotopies (note that not all planar isotopies lift to Legendrian isotopies, for example, a rotation may not lift to a Legendrian isotopy). Hence it is necessary to consider planar isotopies.
    We recall that a planar isotopy can be converted to a Legendrian isotopy if and only if there are no vertical tangencies that appear during the planar isotopy. Planar isotopies can be divided into -- rotations, perturbation of arcs without introducing new self intersection, scaling, and translations. Among these, scaling 
and translations can be lifted to a Legendrian isotopy. However, the rotations (local or global) may introduce vertical tangencies. Further, a \textit{global}  rotation of knot diagram can be 
seen as a composition of finitely many locally supported rotations (see Remark \ref{composition}). Thus, it is \emph{sufficient} to consider local rotations for our purpose. 
Therefore we have the following definition.

\begin{definition}
We say that front projections $F_0$ and $F_1$ differ by an R-move (local rotation) in a disc $Q$ and they are identical outside $Q$ if there is a 
Reidemeister move (local rotation) supported within the disc $Q$ such that after application of it on $F_0$ as a topological knot diagram, $F_0$ changes to topological knot diagram $F_1$. 
\end{definition}

Please note that in the above definition, a knot diagram refers to a regular projection and a front projection refers to a generic front projection where there is no vertical tangency and only one type of crossing allowed. Also note that during this process we add cusps which are then counted in the cost function.

In the following statement, we show that a topological isotopy between two Legendrian knots can be converted to a Legendrian isotopy by the process of stabilizations. It can be articulated in the language of R-moves and LR-moves.

First, we will show that any R-move between two front projections can be converted into an LR-move by stabilizing the fronts. 

\begin{theorem}\label{CostL1}
    Let $F_0$ and $F_1$ be front projections that differ by an R-move or a local rotation in a disc $Q$ and are identical outside $Q$. Then $F_0$ and $F_1$ can be stabilized sufficiently many times within $Q$ to obtain Legendrian isotopic knots and this Legendrian isotopy is supported in $Q$.
\end{theorem}
\begin{proof}
   Reidemeister's Theorem \cite{R} implies that any two knot diagrams representing isotopic knots are related by a sequence of Reidemeister moves and local rotations. Our goal is to convert R-moves and local rotations into LR-moves supported within $Q$ by using stabilizations and destabilizations within $Q$.  In particular, strands involved in any R-move or local rotation within $Q$ differ only in addition or deletion of zig-zags. Thus, we consider each R-move and local rotation one by one. 
   
   By hypothesis $F_0$ and $F_1$ are connected by a single R-move or a local rotation, then we have the following cases:

   \noindent  \textit{(i) Conversion of R1 move into a sequence of LR-moves.}
   
Let $F_0$ and $F_1$ be connected by an R1 move in a disc $Q$ and are identical outside $Q$. Let $x$ be the new crossing in $F_1$ created by the R1 move. Then 
$x$ is either a left-handed crossing or a right-handed crossing. Let us assume that $x$ is a left-handed crossing, the case of a right-handed crossing is similar to it. For an R1 move, we have a single strand in the disc. 
This strand could be vertical or non-vertical. Figure \ref{fig:CostLR1} shows two examples (one vertical and one non-vertical) of such front projections 
$F_0$ and $F_1$ in the disc $Q$.  All the other possible cases are obtained from these by stabilizing and destabilizing (see Figure \ref{fig:CostLR1}) arcs $\alpha$, $\beta$, and $\gamma$.
    \begin{figure}
       \centering
       \includegraphics[width=0.15\textwidth]{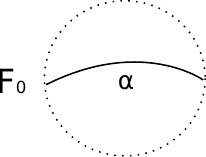}
       \hspace{30pt}
       \includegraphics[width=0.15\textwidth]{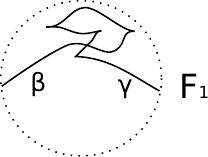}
       \hspace{30pt}
       \includegraphics[width=0.15\textwidth]{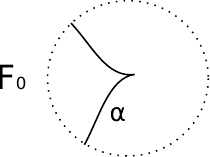}
        \hspace{30pt}
       \includegraphics[width=0.15\textwidth]{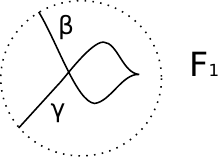}
       \caption{Basic cases of front projections $F_0$ and $F_1$ connected by an R1 move.}
       \label{fig:CostLR1}
   \end{figure}
   Since stabilization can be moved across crossings without changing the Legendrian knot type, it is enough to discuss cases shown in Figure \ref{fig:CostLR11} and Figure \ref{fig:CostLR12}. For the case shown in Figure \ref{fig:CostLR11}, the front projection $F$ obtained by adding two cusps to $F_0$ is connected to $F_1$ by an LR2 move followed by an LR1 move. This Legendrian isotopy connecting $F$ and $F_1$ is supported within $Q.$ In the second case when $F_0$ and $F_1$ are connected by R1 move as shown in Figure \ref{fig:CostLR12}, the Legendrian isotopy supported within $Q$ is obtained by stabilizing $F_0$ three times. 
    \begin{figure}
        \centering
        \includegraphics[scale=0.5]{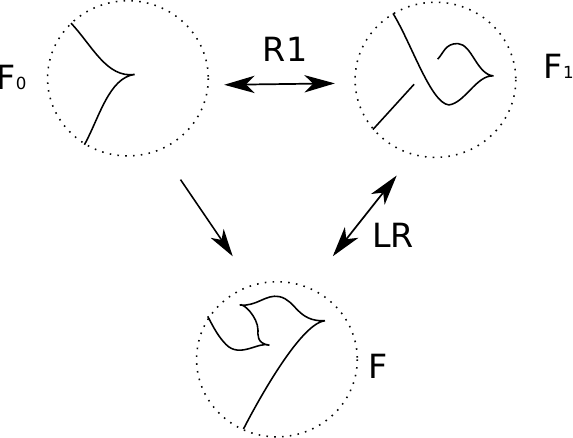}
        \caption{Converting an R1 move into a sequence of LR-moves.}
        \label{fig:CostLR11}
    \end{figure}
    
    \begin{figure}
        \centering
        \includegraphics[scale=0.5]{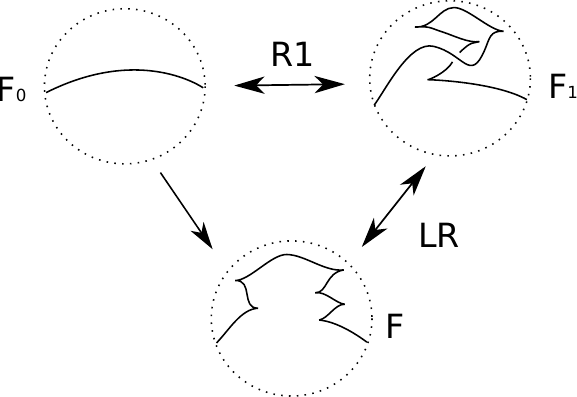}
        \caption{Converting an R1 move into a sequence of LR-moves.}
        \label{fig:CostLR12}
    \end{figure}

 \noindent  \textit{(ii) Conversion of R2 move into a sequence of LR-moves.}

 Consider front projections $F_0$ and $F_1$ that are connected by an R2 move in a disc $Q$ and are identical outside $Q$. There are only two cases up to Legendrian isotopy, either the two strands are vertical or non-vertical within $Q$. 
 Therefore, we consider two cases as shown in Figure \ref{fig:CostLR21} and Figure \ref{fig:CostLR22}. All possible front projections $F_0$ and $F_1$ that 
 are connected by a horizontal R2 move can be obtained by stabilizing $F_0$ and $F_1$ shown in Figure \ref{fig:CostLR22} and the cases for the vertical R2 
 move can be exhausted by stabilizing the fronts shown in Figure \ref{fig:CostLR21}. Both of these figures prove the existence of Legendrian isotopies between $F_1$ and $F$ supported within $Q$.
    
\begin{figure}
        \centering
        \includegraphics[scale=0.5]{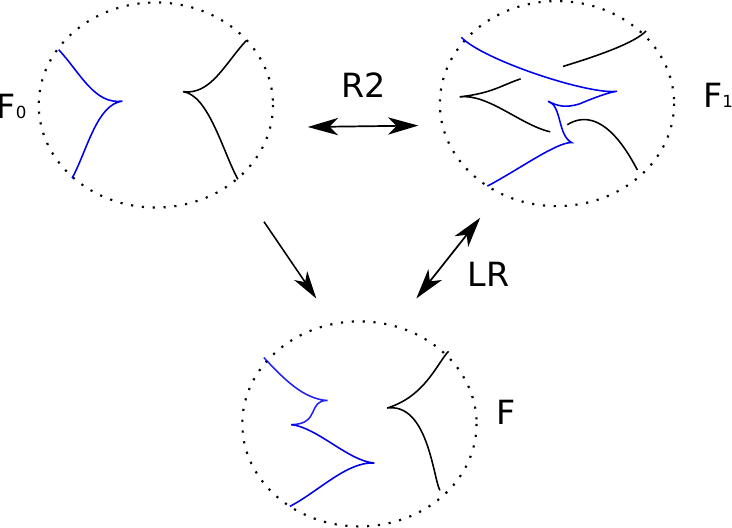}
        \caption{Converting an R2 move into a sequence of LR-moves.}
        \label{fig:CostLR21}
    \end{figure}

    \begin{figure}
        \centering
        \includegraphics[scale=0.5]{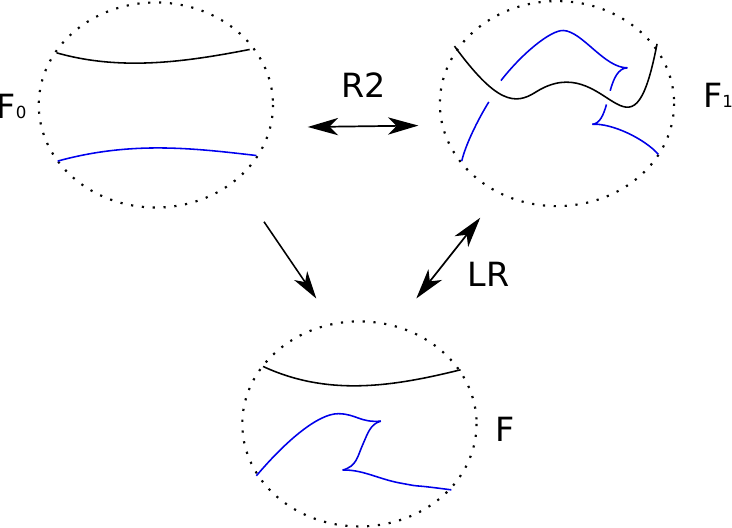}
        \caption{Converting an R2 move into a sequence of LR-moves.}
        \label{fig:CostLR22}
    \end{figure}

 \noindent  \textit{(iii) Conversion of R3 move into a sequence of LR-moves.}

  Now we consider the R3 move. An R3 move involves three arcs (say blue, red and black) and three crossings formed by these arcs. In an R3 move 
  we slide an arc which is under (or over) the other two arcs across a crossing. Depending on the configuration of each crossing we get several diagrams connected by R3 move. 
  For example in Figure \ref{fig:CostLR32} the black arc is under both the blue and
the red arcs. One could also consider a case where the black arc is over the other 
  two arcs or the case where blue arc is under the red arc. Also, note that, in Figure \ref{fig:CostLR32}, topologically, sliding black arc is equivalent to 
  sliding blue arc. Therefore, counting up to rotation of diagrams, there are the following $4$ cases of topological knot diagrams connected by R3 move: 
  (i) blue arc is overcrossing and the black arc is undercrossing, (ii) blue arc is undercrossing and the black arc is overcrossing, (iii) black arc is overcrossing and the red arc is undercrossing, (iv) black arc is undercrossing and the red arc is overcrossing. In the 
  following discussion we only consider front projections connected by R3 moves corresponding to the case (i) since other cases work in a similar manner.

Let $F_0$ and $F_1$ be front projections that are connected by R3 move in $Q$ and are identical outside $Q$. As we have seen in the previous cases, any front diagram will differ by stabilizations and 
destabilizations of strands within $Q$, the rest of the possible cases of front projections $F_0$ and $F_1$ are obtained by a sequence of stabilizations and destabilizations within $Q$. One can give a similar argument, as given for R1 and R2 moves, using strands to justify why are these the only possible cases. Figure \ref{fig:CostLR31} and Figure \ref{fig:CostLR32} illustrate two such cases of 
$F_0$ and $F_1$. The front projections $F_0$ and $F_1$ shown in Figure \ref{fig:CostLR31} are connected by a sequence of LR-moves within the disc $Q$. Also, it can be checked that there exists a Legendrian isotopy supported within $Q$ connecting $F_0$ and $F_1$ shown in Figure \ref{fig:CostLR32}.

 \noindent  \textit{(iv) Conversion of a local rotation into a sequence of LR-moves.}
 Note that certain planar isotopies allowed in the topological category are not allowed in the Legendrian category. Hence we need to consider the case of planar isotopies as well.
 Now consider the case when front projections $F_0$ and $F_1$ are connected by a rotation in a disc $Q$ and identical outside $Q$. All front projections connected by such local planar isotopy can be obtained by stabilizing the front projections shown in Figure \ref{fig:CostP}. The bottom two front projections shown in Figure \ref{fig:CostP} are connected by a sequence of LR2 moves within the disc.
\begin{figure}
        \centering
        \includegraphics[scale=0.5]{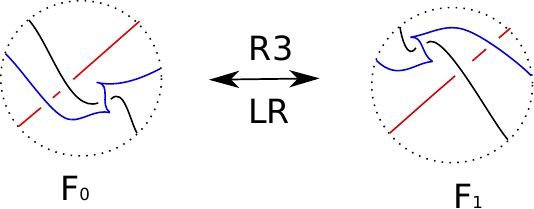}
        \caption{This is one of the cases where an R3 move is an LR3 move.}
        \label{fig:CostLR31}
    \end{figure}

    \begin{figure}
        \centering
        \includegraphics[scale=0.5]{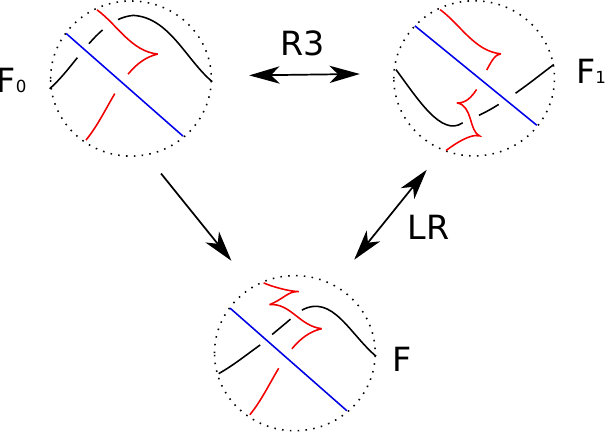}
        \caption{Converting an R3 move into a sequence of LR-moves.}
        \label{fig:CostLR32}
    \end{figure}
    
\begin{figure}
        \centering
        \includegraphics[scale=0.5]{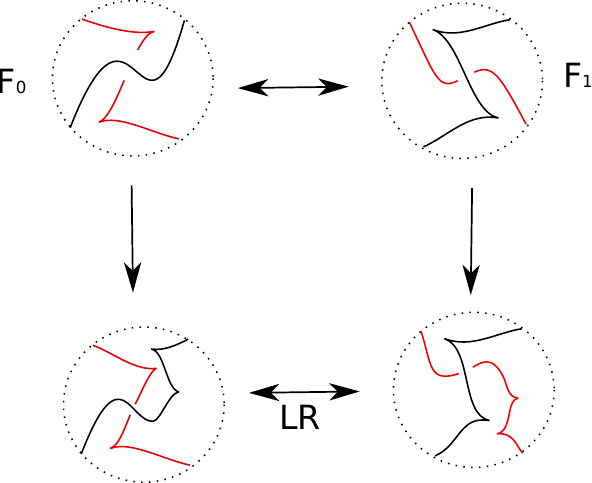}
        \caption{Converting a local rotation into a sequence of LR-moves.}
        \label{fig:CostP}
    \end{figure}

Hence, $F_0$ and $F_1$ can be stabilized sufficiently within $Q$ to obtain Legendrian isotopic knots and this Legendrian isotopy is supported in $Q$.
\end{proof}
\begin{remark} 
\label{finite}
When converting a smooth isotopy to a Legendrian isotopy, we need to add a cusp when we introduce a vertical tangency and add two cusps every time we introduce a crossing that is not allowed in the front projection. From the proof, it is clear that we need finitely many cusps to convert each (smooth) isotopy to Legendrian isotopy.
\end{remark}

\begin{remark}\label{composition}
    Note that two front projections that are connected by a global rotation of knot diagrams can be seen as a composition of planar rotations near crossings and translations of the diagram as illustrated in Figure \ref{fig:costglob1} and Figure \ref{fig:costglob2}.
    \begin{figure}
        \centering
        \includegraphics[scale=0.8]{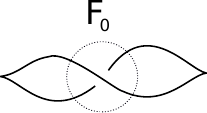}
        \includegraphics[scale=0.8]{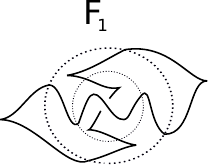}
        \includegraphics[scale=0.8]{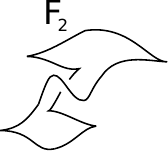}
        \caption{Topologically the projection $F_2$ is obtained by $90^\circ$ rotation of $F_0$. The front projection $F_1$ is obtained from $F_0$ by a local rotation done by adding cusps in a disc near the crossing. Then $F_2$ is obtained by stabilizing $F_0$.}
        \label{fig:costglob1}
    \end{figure}
   
   \begin{figure}
       \centering
       \includegraphics[scale=0.8]{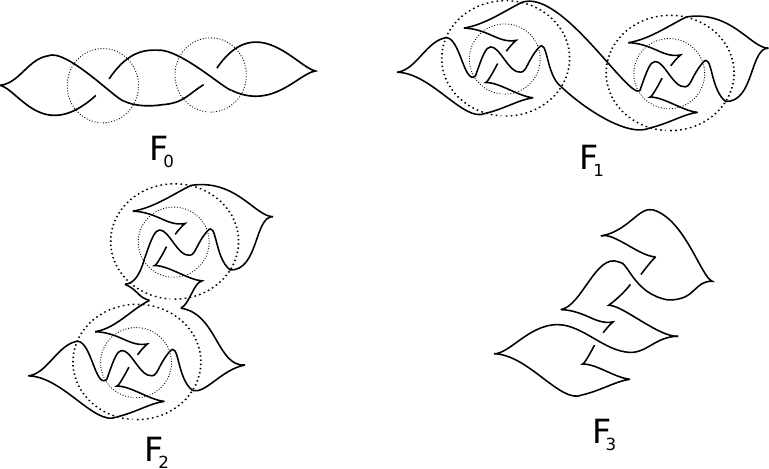}
       \caption{Achieving a global rotation $F_3$ of $F_0$ via a sequence of local planar rotations in discs around crossings in $F_1$.}
     \label{fig:costglob2}
   \end{figure}
\end{remark}

\begin{remark}
    Let $K_0$ and $K_1$ be Legendrian knots with front projections $F_0$ and $F_1$ respectively. The notations $K_0=K_1$ and $F_0=F_1$ will be used to denote that $K_0$ and $K_1$ are of the same Legendrian knot type.
\end{remark}
Let $S^{\pm}_{Q}(F_0)$ denote the front projection obtained by positively or negatively stabilizing $F_0$ within the disc $Q$ respectively. Then we give the following definition for the cost between two front projections in the same topological knot type.

\begin{definition}
    Let $F_0$ and $F_1$ be front projections that differ by an R-move within a disc $Q$ and are identical outside $Q$. Then define the Cost function as follows:
    \begin{align*}
        \Cost(F_0,F_1) := \min  \{m_0+n_0+m_1+n_1:  (S^+_{Q})^{m_0}(S^-_{Q})^{n_0}(F_0)
       =(S^+_{Q})^{m_1}(S^-_{Q})^{n_1}(F_1)\}.
    \end{align*}
        
\end{definition}
\begin{remark}
    The definition makes sense in light of Theorem \ref{CostL1}.
\end{remark}
\begin{Lemma}\label{CostL2}
    Let $F$ and $\Tilde{F}$ be front projections of Legendrian knots of the same topological knot type. Then there exists a finite sequence $SF=\left(F=F_0,F_1,\ldots,F_n,F_{n+1}=\Tilde{F}\right)$ of front projections and discs $Q_i$'s such that $F_i$ and $F_{i+1}$ are identical outside $Q_i$ and are connected by an R-move or local planar isotopies inside $Q_i\: \forall \ 0\leq i \leq n-1,$ for some positive integer $n$. Moreover, $F_n $ and $F_{n+1} $ are connected by stabilizations and destabilizations inside $Q_n$.
\end{Lemma}
\begin{proof}
    There exists a sequence of R-moves and local rotations connecting $F$ to $\Tilde{F}$ by Reidemeister's theorem \cite{R} since $F$ and $\Tilde{F}$ are in the same topological knot 
    type. Fix a sequence of R-moves and local rotations connecting $F$ to $\Tilde{F}$. Denote it by $\mathcal{S}$. 

 Let $D_0=F, D_1,\ldots,D_n=\Tilde{F}$ be a sequence 
 of topological knot diagrams obtained in the sequence $\mathcal{S}$. The knot diagram $D_0$ is identical to $F$. Let $Q_i$ be the disc where $D_i$  and $D_{i+1}$ differ by an R-move or local rotation and are identical outside $Q_i\: \forall \ 0\leq i\leq n-1$. 

Using $D_0,D_1,\ldots,D_n$, construct a sequence $F_0,F_1,\ldots,F_n,F_{n+1}=\Tilde{F}$ such that $F_{i}$ and $F_{i+1}$ are identical outside $Q_{i}\: \forall \ 0\leq i \le n$. Construct $F_1$, from $D_0$ and $D_1$ as follows: Fix $F_1$ to be identical to $D_0$ outside $Q_0$. Inside $Q_0$, fix $F_1$ to be a front projection 
obtained by converting $D_1$ into a front projection without changing the topological knot type as discussed earlier (Refer to Figure \ref{fig:CTF}). Note that inside $Q_1$, we are making a 
choice for $F_1$ as there are many ways to convert a topological knot diagram into a front projection. The front projection $F_1$ obtained by this process is 
identical to $F_0$ outside $Q_0$ and connected to $F_0$ by an R-move or local rotation inside $Q_0$. Repeat the process inside and outside the disc $Q_1$ to construct $F_2$ from $F_1$ and $D_2$. Repeating this process $n$-times we get a 
desired sequence $F_0, F_1,\ldots, F_n$. The front projection $F_n$ is identical to $F_{n-1}$ outside $Q_{n-1}$ and differs from it by an R-move or local rotation inside $Q_{n-1}$. Note that by construction the front projection $F_n$ is 
obtained from $D_n$ and topologically $D_n$ and $\Tilde{F}$ are planar isotopic. Thus, $F_n$ differs from $\Tilde{F}$ only in terms of number of zig-zags. After moving all the zig-zags in $F_n$ and $\Tilde{F}$ to 
a common disc $Q_{n}$, we can assume that $F_n$ and $\Tilde{F}$ are identical outside $Q_{n}$ and differ by some stabilizations in $Q_{n}$. Hence, we have 
the desired sequence $F_0,F_1,\ldots,F_n,F_{n+1}=\Tilde{F}$ of front projections. 
   
\end{proof}

Let $\mathcal{SF}(F,\Tilde{F})$ denote the set of all the sequences of front projections connecting $F$ and $\Tilde{F}$ with the properties mentioned in Lemma \ref{CostL2}. Then we give the following definition: 
\begin{definition}
 Let $F$ and $\Tilde{F}$ be front projections of Legendrian knots of the same topological knot type. Let $SF=\left(F=F_0,F_1,\ldots,F_{n+1}=\Tilde{F}\right)$ be an element from $ \mathcal{SF}(F,\Tilde{F})$. Define

 \begin{align*}
     \Cost(F,\Tilde{F},SF)&:=\sum_{i=0}^{n}\Cost(F_{i},F_{i+1}), \quad \text{and}\\
     \Cost(F,\Tilde{F}):=\min\{&\Cost(F,\Tilde{F},SF):SF\in \mathcal{SF}(F,\Tilde{F})\}.
 \end{align*}
\end{definition}
Note that $\Cost(F,\Tilde{F})$ is well-defined.
\begin{remark}\label{re:Cost1}
    From the definition it can be seen that the Cost function is symmetric, that is, $\Cost(F,\Tilde{F})=\Cost(\Tilde{F},F)$ for all front projections $F$ and $\Tilde{F}$ in the same topological knot type.
\end{remark}

\begin{Lemma}\label{CostL3}
   Let $F$ and $\Tilde{F}$ be front projections of Legendrian knots of the same topological knot type. Let $F_1$ and $\Tilde{F}_1$ be another front projections Legendrian isotopic to $F$ and $\Tilde{F}$, respectively. Then $\Cost(F,\Tilde{F})=\Cost(F_1,\Tilde{F_1})$.
\end{Lemma}
\begin{proof}
    Let $S$ and $\Tilde{S}$ be sequences of LR-moves connecting $F_1$ to $F$ and $\Tilde{F}$ to $\Tilde{F}_1$, respectively. Then $S\in \mathcal{SF}(F_1, F)$ and $S_1\in\mathcal{SF}(\Tilde{F},\Tilde{F}_1)$. Let $SF$ be an arbitrary element of $ \mathcal{SF}(F,\Tilde{F})$. Let $X=\left(S,SF,S_1\right)$, then $X\in\mathcal{SF}(F_1,\Tilde{F}_1)$ and $\Cost(F_1,\Tilde{F}_1,X)=\Cost(F,\Tilde{F},SF)$. Thus we have $\Cost(F_1,\Tilde{F}_1)\leq \Cost(F,\Tilde{F},SF) \: \forall SF\in \mathcal{SF}(F,\Tilde{F})$. Hence, $\Cost(F_1,\Tilde{F}_1)\leq \Cost(F,\Tilde{F})$.

    Similarly, it can be shown that $\Cost(F,\Tilde{F})\leq \Cost(F_1,\Tilde{F}_1)$. Hence, $\Cost(F,\Tilde{F})=\Cost(F_1,\Tilde{F}_1)$.
\end{proof}
In light of the above Lemma, we can introduce the Cost function for the \emph{Legendrian knot type}. 
\begin{definition}
    Let $K$ and $\Tilde{K}$ be Legendrian knots of the same topological knot type. Define $\Cost(K,\Tilde{K}):= \Cost(F,\Tilde{F}),$ where $F$ and $\Tilde{F}$ are front projections of $K$ and $\Tilde{K}$, respectively.
\end{definition}

\begin{example}
Let $F_0$, $F_1$ and $F_2$ be the Legendrian unknots as shown in Figure \ref{CU}, then $\Cost(F_0,F_1)=1,$ $\Cost(F_1,F_2)=1$ and $\Cost(F_0,F_2)=2.$ As we will see in Lemma \ref{CostL8}, $\Cost(F_0,F_2)\geq 2$
    
    \begin{figure}[!htbp]
       \centering
       \includegraphics[width=0.5\textwidth]{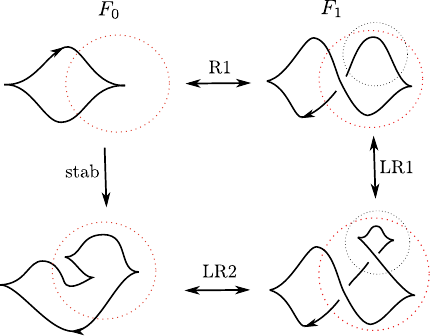}
       \hspace{30pt}
       \includegraphics[width=0.7\textwidth]{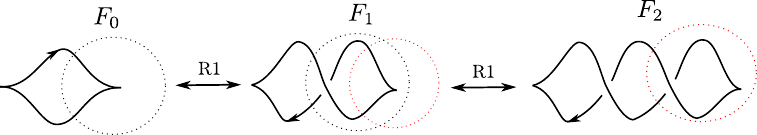}
       \caption{Cost associated with two Legendrian unknots.}
       \label{CU}
   \end{figure}
\end{example}

\begin{theorem}\label{CostL4}
Cost$(K,\Tilde{K})=0$ if and only if $K$ is Legendrian isotopic to $\Tilde{K}$.
\end{theorem}
\begin{proof}
     Let $F$ and $\Tilde{F}$ denote the front projections of  $K$ and $\Tilde{K}$, respectively. If $K$ and $\Tilde{K}$ are Legendrian isotopic then there exists a sequence $SF_0$ of LR-moves connecting $F$ to $\Tilde{F}$. Therefore, $\Cost(F,\Tilde{F})\leq \Cost(F,\Tilde{F},SF_0)=0$. Hence, $\Cost(K,\Tilde{K})=0$.

     Now, assume $\Cost(K,\Tilde{K})=0$. Then there exists a sequence $SF_0\in \mathcal{SF}(F,\Tilde{F})$ given by $(F=F_0,F_1,\cdots,F_n,F_{n+1}=\Tilde{F})$ such that the $\Cost(F,\Tilde{F},SF_0)=0$. Since, $\Cost(\cdot,\cdot)$ is a non-negative integer-valued function, we get $\Cost(F_{i},F_{i+1})=0 \: \forall 0\leq i\leq n$. Thus $F_i$ is Legendrian isotopic to $F_{i+1}$ $\forall 0\leq i\leq n$. Hence, $F$ is Legendrian isotopic to $\Tilde{F}.$ 
\end{proof}

\begin{Lemma}\label{CostL5}
     Let $F$ and $\Tilde{F}$ be front projections of the same topological knot type. Then there exist positive integers $p,q,r,s$ such that  $\Cost(F,\Tilde{F})=p+q+r+s$ and $(S^+)^p(S^-)^r(F)=(S^+)^q(S^-)^s(\Tilde{F})$. 
\end{Lemma}   
\begin{proof}
  Let $\Cost(F,\Tilde{F})=\Cost(F,\Tilde{F},SF_0)$ for some $SF_0\in \mathcal{SF}(F,\Tilde{F})$. Let $SF_0$ be given by $\left(F=F_0,F_1,\cdots,F_n,F_{n+1}=\Tilde{F}\right)$. Let $Q_{i}$ be the disc such that $F_i$ and $F_{i+1}$ differ by an R-move or local planar isotopy within $Q_i$, $ 0\leq i \leq n$. Then there exist positive integers $p_i,q_i,r_i,s_i$ such that $\Cost(F_i,F_{i+1})=p_{i}+q_{i}+r_i+s_i$ and $(S^+_{Q_i})^{p_{i}}(S^-_{Q_i})^{r_i}(F_i)=(S^+_{Q_i})^{q_{i}}(S^-_{Q_i})^{s_{i}}(F_{i+1})$ for all $ 0\leq i \leq n$ and $\Cost(F,\Tilde{F})=\sum_{i=0}^n(q_{i}+p_{i}+r_i+s_i)$. Now we have
  \begin{align*}
      &(S^+)^{p_n}(S^-)^{r_n}(S^+)^{p_{n-1}}(S^-)^{r_{n-1}}\cdots (S^+)^{p_1}(S^-)^{r_1}(S^+)^{p_{0}}(S^-)^{r_0}(F)\\
      &=(S^+)^{p_{n}}(S^-)^{r_n}(S^+)^{p_{n-1}}(S^-)^{r_{n-1}}\cdots (S^+)^{p_{1}}(S^-)^{r_1} (S^+)^{q_0}(S^-)^{s_0}(F_1)\\
      &=(S^+)^{q_0}(S^-)^{s_0}(S^+)^{p_{n}}(S^-)^{r_n}(S^+)^{p_{n-1}}(S^-)^{r_{n-1}}\cdots (S^+)^{p_{1}}(S^-)^{r_1}(F_1)\\
      &=(S^+)^{q_0}(S^-)^{s_0}(S^+)^{p_{n}}(S^-)^{r_n}(S^+)^{p_{n-1}}(S^-)^{r_{n-1}}\cdots (S^+)^{q_{1}}(S^-)^{s_1}(F_2)\\
      &\vdots \\
      &=(S^+)^{q_0}(S^-)^{s_0}(S^+)^{q_1}(S^-)^{s_1}\cdots (S^+)^{q_{n}}(S^-)^{s_n}(\Tilde{F}).
  \end{align*}
  Then $p=\sum_{i=0}^np_i, q=\sum_{i=0}^nq_i, r=\sum_{i=0}^nr_i$ and $s=\sum_{i=0}^ns_i$ are the desired integers.
\end{proof}
\begin{proposition}\label{minpropn}
For two front projections $F$ and $\Tilde{F}$ of the same topological knot type, $$\Cost(F,\Tilde{F})=\min\{m+n+p+q:(S^+)^m(S^-)^n(F)=(S^+)^p(S^-)^q(\Tilde{F})\}.$$
\end{proposition}
\begin{remark}
    The above lemma gives a \emph{quantitative} version of the theorem of Fuchs and Tabachnikov.
\end{remark}

Before proving this Proposition, we need the following Lemma.

\begin{Lemma}\label{CostL6}
    Let $F$ and $\Tilde{F}$ be front projections of the same topological knot type. Let $p,n,\Tilde{p},\Tilde{n}$ be positive integers such that $(S^+)^p(S^-)^n(F)$ is Legendrian isotopic to  $(S^+)^{\Tilde{p}}(S^-)^{\Tilde{n}}(\Tilde{F})$. Then $\Cost(F,\Tilde{F})\leq p+n+\Tilde{p}+\Tilde{n}.$
\end{Lemma}
\begin{proof}
    Let $F_1$ be a front projection obtained by applying $p$ positive and $n$ negative stabilizations on an arc $\alpha$ of $F$ in a disc $Q_0$. Let $F_2$  be a front projection obtained by applying $\Tilde{p}$ positive and $\Tilde{n}$ negative stabilizations on an arc $\beta$ of $\Tilde{F}$ in a disc $Q_2$. Then $\Cost(F,F_1)=p+n$ and $\Cost(F_2,\Tilde{F})=\Tilde{p}+\Tilde{n}$ since $(S^+_{Q_0})^{p}(S^-_{Q_0})^n(F)= F_1$ and $F_2= (S^+_{Q_2})^{\Tilde{p}}(S^-_{Q_2})^{\Tilde{n}}(\Tilde{F})$. Since stabilization is independent of its location in a front projection $F_1$ is Legendrian isotopic to $(S^+)^p(S^-)^n(F)$ and $F_2$ is Legendrian isotopic to $(S^+)^{\Tilde{p}}(S^-)^{\Tilde{n}}(\Tilde{F})$ which implies that $F_1$ and $F_2$ are Legendrian isotopic. Let $X$ be a sequence of LR-moves connecting $F_1$ to $F_2$ and let $(X)$ denote the sequence of front projections obtained by applying $X$ to $F_1$. Then the sequence $SF_0:=\left(F,(X),\Tilde{F}\right) \in \mathcal{SF}(F,\Tilde{F})$ and
    \begin{align*}
\Cost(F,\Tilde{F},SF_0)&=\Cost(F,F_1)+\Cost(F_2,\Tilde{F})+\Cost(F_1,F_2,(X))\\
&=p+n+\Tilde{p}+\Tilde{n}+0.
    \end{align*}
    Hence, $\Cost(F,\Tilde{F})\leq p+n+\Tilde{p}+\Tilde{n}. $ 
\end{proof}
\begin{proof}[Proof of Proposition \ref{minpropn}]

The Lemma \ref{CostL5} gives the existence of such $m,n, p$ and $q$. The Lemma \ref{CostL6} gives the upper bound for the same. From Lemmas \ref{CostL5}, \ref{CostL6} our proposition follows.
\end{proof}

\begin{Lemma}\label{CostL7}
If $tb(K)=tb(\Tilde{K})$, then $\Cost(K,\Tilde{K})$ is even. 
\end{Lemma}
\begin{proof}
   Let $n$ and $\Tilde{n}$ be the minimum number of stabilizations required on $K$ and $\Tilde{K}$ to obtain Legendrian isotopic knots $K'$ and $\Tilde{K}'$ respectively. Then $tb(K)-n =tb(K')=tb(\Tilde{K}')=tb(\Tilde{K})-\Tilde{n}$. Hence, $n=\Tilde{n}$ and $\Cost(K,\Tilde{K})=2n.$
\end{proof}

\begin{Lemma}\label{CostL8}
    Let $K$ and $\Tilde{K}$ be Legendrian knots of the same topological knot type. Then $$\Cost(K,\Tilde{K})\geq \max\{\abs{tb(K)-tb(\Tilde{K})},\abs{rot(K)-rot(\Tilde{K})}\}.$$
\end{Lemma}
\begin{proof}
    Let $\Cost(K,\Tilde{K})=p+n+\Tilde{p}+\Tilde{n}$, where $(S^+)^{p}(S^-)^{n}(K)=(S^+)^{\Tilde{p}}(S^-)^{\Tilde{n}}(\Tilde{K})$. Then $tb(K)-p-n=tb(\Tilde{K})-\Tilde{p}-\Tilde{n}$ and $rot(K)+p-n=rot(\Tilde{K})+\Tilde{p}-\Tilde{n}$. Since $p,n,\Tilde{p},$ and $ \Tilde{n}$ are non-negative integers, $p+n+\Tilde{p}+\Tilde{n}\geq |(p+n)-(\Tilde{p}+\Tilde{n})|=\abs{tb(K)-tb(\Tilde{K})}$ and $p+n+\Tilde{p}+\Tilde{n}\geq |(\Tilde{p}+n)-(p+\Tilde{n})|=\abs{rot(K)-rot(\Tilde{K})}$. Hence, $$\Cost(K,\Tilde{K})\geq \max\{\abs{tb(K)-tb(\Tilde{K})},\abs{rot(K)-rot(\Tilde{K})}\}.$$

\end{proof}
\begin{remark}\label{NewRemark2}
 Note that if $K=(S^+)^m(S^-)^n(L)$ then $\Cost(K,L)=m+n$, where $m$ and $n$ are positive integers. For the justification, we clearly have $
   \Cost(K,L)\leq m+n$ and $tb(K)=tb(L)-m-n$. Now, by Lemma \ref{CostL8}, we have $\Cost(K,L)\geq |tb(K)-tb(L)|= m+n$.
\end{remark}
\begin{Lemma}\label{CostL9}
    Let $\mathcal{K}$ be a topological knot type which is Legendrian simple. Let $K$ and $\Tilde{K}$ be Legendrian representatives of $\mathcal{K}$. Then 
$$\Cost(K,\Tilde{K})=\max\{\abs{tb(K)-tb(\Tilde{K})},\abs{rot(K)-rot(\Tilde{K})}\}.$$
\end{Lemma}
    
\begin{proof}
 Let $t$ and $\Tilde{t}$ be Thurston-Bennequin numbers of $K$ and $\Tilde{K}$, respectively. Let $r,\Tilde{r}$ be the rotation number invariants of $K$ and $\Tilde{K}$ respectively. Let $p=\abs{t-\Tilde{t}}$ and $m=\abs{r-\Tilde{r}}$.
Using Lemma \ref{CostL8}, we get $\Cost(K,\Tilde{K})\geq \max\{p,m\}$. Now we need to prove the reverse inequality. Note that
\begin{equation}\label{Costeq3}
    p+m=0 \text{ mod }2,
\end{equation} 
since $tb(K)+rot(K)=1 \text{ mod }2$ for every Legendrian knot $K$ (Proposition 3.5.23, \cite{GCT}). Without loss of generality, we can assume that $t\geq\Tilde{t}$. Now we will give the proof when $r\leq \Tilde{r}.$ The case $r\geq \Tilde{r}$ is similar to this. Now we consider the following cases.

\noindent \textbf{Case 1:} $p=0$.

From Equation \ref{Costeq3}, we have $m=2n$ for some integer $n.$

\noindent \textbf{Subcase 1:} $m=0$.

Since $\mathcal{K}$ is a Legendrian simple knot type, $K$ is Legendrian isotopic to $\Tilde{K}$ and $\Cost(K,\Tilde{K})=0.$

\noindent \textbf{Subcase 2:} $m>0$.
 Let $K'$ and $\Tilde{K}'$ be the Legendrian knots obtained by $n$ positive stabilizations on $K$ and $n$ negative stabilizations on $\Tilde{K}$, respectively. Then $tb(K')=tb(\Tilde{K}')$ and $rot(K')=rot(\Tilde{K}')$. Therefore, $K'$ is Legendrian isotopic to $\Tilde{K}'$. Hence, $\Cost(K,\Tilde{K})\leq 2n=m$.

\noindent \textbf{Case 2:} $p\ne 0$.

\noindent \textbf{Subcase 1:} $m=0$.

Using Equation \ref{Costeq3}, we have $p=2d$ for some positive integer $d$. Let $\hat{K}$ be the Legendrian knot obtained by stabilizing $K$ positively and negatively $d$ times each. Then $tb(\hat{K})=\Tilde{t}, rot(\hat{K})=\Tilde{r}$. Thus $\hat{K}$ is Legendrian isotopic to $\Tilde{K}$ since $\mathcal{K}$ is a Legendrian simple knot type. Hence, $\Cost(K,\Tilde{K})\leq 2d=p$.

\noindent \textbf{Subcase 2:} $p=m$.

 Let $K_2$ be the Legendrian knot obtained after positively stabilizing $K$, $m$ times. Then $rot(K_2)=r+m=\Tilde{r}$ and $tb(K_2)=t-m=\Tilde{t}$. Hence, $K_2$ is Legendrian isotopic to $\Tilde{K}$ and $\Cost(K,\Tilde{K})\leq m$.

\noindent \textbf{Subcase 3:} $p>m>0$. 

 Let $K_3$ be the Legendrian knot obtained by $m+\frac{p-m}{2}$ positive stabilizations and $\frac{p-m}{2}$ negative stabilizations on $K$. Then $tb(K_3)=t-p=\Tilde{t}$ and $rot(K_3)=r+m=\Tilde{r}$. Thus $K_3$ is Legendrian isotopic to $\Tilde{K}$. Hence, $\Cost(K,\Tilde{K})\leq p$.

\noindent\textbf{Subcase 4:} $m>p>0$.

Let $K_4$ be the Legendrian knot obtained by $ p+\frac{m-p}{2}$ positive stabilizations on $K$. Let $K_5$ be the Legendrian knot obtained by $\frac{m-p}{2}$ negative stabilizations on $\Tilde{K}$. Then $tb(K_4)=t-p-\frac{m-p}{2}=\Tilde{t}-\frac{m-p}{2}=tb(K_5)$ and $rot(K_4)=r+p+\frac{m-p}{2}=\Tilde{r}-\frac{m-p}{2}=rot(K_5)$. Thus $K_4$ and $K_5$ are Legendrian isotopic. Hence, $\Cost(K,\Tilde{K})\leq p+\frac{m-p}{2}+\frac{m-p}{2}=m.$ 

\end{proof}

\section{Cost Associated With some special families of knots}

In this section, we consider a family of oriented Legendrian knots $E_{k,l}$ whose front projection has $k\ge 1$ crossings on the left and $l\geq 1$ on the right as shown in Figure \ref{fig:ekl}. 

\begin{figure*}[!htbp]
    \centering
        \includegraphics[height=1.2in]{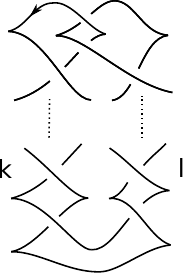}
    \caption{Front projection of $E_{k,l}$.}
    \label{fig:ekl}
\end{figure*}

\begin{Lemma}\label{CostL10}
    Let $k,l\:\geq1$ and $n\geq4$ be positive integers with $k+l=n$. Then $\Cost(E_{k,l},E_{k+1,l-1})=2$ if $k\neq l-1$ and $0$ otherwise.
\end{Lemma}
     
\begin{proof}
It is proved in \cite{NEV} that $E_{k,l}$ and $E_{k+1,l-1}$ are Legendrian isotopic if and only if $k=l-1$. Then $\Cost(E_{k,l},E_{k+1,l-1})=0 $, if $k=l-1$ and nonzero, otherwise. Below we have a sequence of LR-moves on one-time stabilized $E_{k,l}$ and $E_{k+1,l-1}$. Hence, $\Cost(E_{k,l},E_{k+1,l-1})=2$.

    \begin{figure}[!htbp]
    \centering
    \includegraphics[scale=0.6]{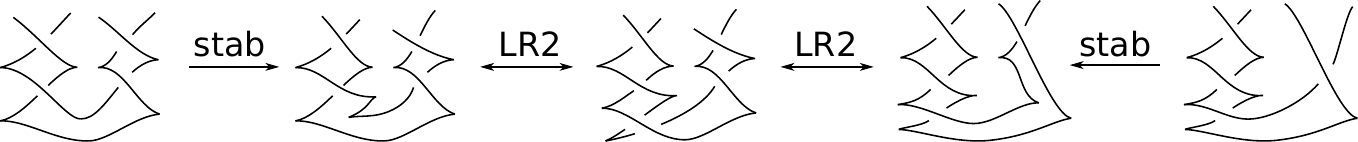}
    \caption{Sequence of LR-moves taking a stabilized $E_{k,l}$ to a stabilized $E_{k+1,l-1}$.}
    \label{fig:ekl1}
\end{figure}
\end{proof}
\begin{remark}{ 
    Proposition 3.6 in \cite{NEV} implies that the Cost between any two maximal $tb$ representatives of the twist knot with $n$-half twists can not exceed $4$. Using the proof from \cite{NEV} one can compute the exact Cost between any two Legendrian representatives of a given twist knot.} 
\end{remark}
\begin{example}
\label{2-b}
   A knot of the form shown in Figure \ref{fig:bridge} is called a $2$-bridge knot and is denoted by $[a_1,a_2,\cdots, a_{2n+1}]$. Here the box with label $a_i$ contains $|a_i|$ positive half twists if $a_i\geq1$ and $|a_i|$ negative half twists if $a_i\leq 1$. Let $L[2n,k:l,2q+1]$ be the Legendrian realization of the knot $[2n,-(k+l),2q+1],\: n,q,k,l\geq 1$ as shown in Figure \ref{fig:bridge1}.
   \begin{figure}[!htbp]
        \centering
        \includegraphics[scale=0.7]{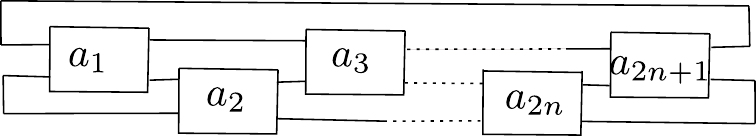}
        \caption{2-bridge knot $[a_1,a_2,\cdots,a_{2n+1}]$}
        \label{fig:bridge}
    \end{figure}
    It can be shown that $rot(L[2n,k:l,2q+1])=0$ and $tb(L[2n,k:l,2q+1])=2q+2n-1$ if $k+l$ is even and is $2q-2n-1$ if $k+l$ is odd. The Legendrian realisation $L[2n,k:l,2q+1]$ maximizes the $tb$ since 
    it admits ungraded rulings \cite{DRD}. Thus, $\text{Cost}(L[2n,k:l,2q+1], L[2n,k+1:l-1,2q+1])\leq 2$ using the technique shown in Figure 
    \ref{fig:ekl1}. Note that $L[2n,k:k,2q+1]$ admits more graded rulings than $L[2n,k+p:k-p,2q+1]$ where $p \geq 1$. Therefore, $L[2n,k:k,2q+1]$ is not 
    Legendrian isotopic to $L[2n,k+p:k-p,2q+1]$ for all $1\leq p \leq k-1$ as depicted in Figure \ref{fig:ruling} for the case when $n=1,q=1$. Thus $
    \text{Cost}(L[2n,k:k,2q+1],L[2n,k+1:k-1,2q+1])=2$.
    
    \begin{figure}[!htbp]
        \centering
        \includegraphics[scale=0.5]{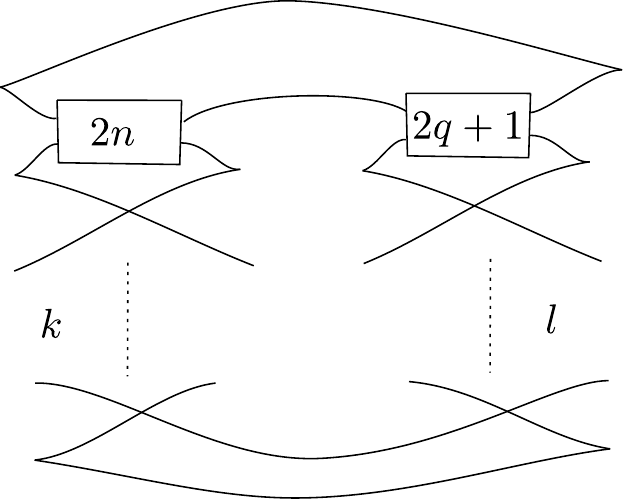}
        \caption{Legendrian knot $L[2n,k:l,2q+1]$}
        \label{fig:bridge1}
    \end{figure}
    \begin{figure}
        \centering
        \includegraphics[scale=0.45]{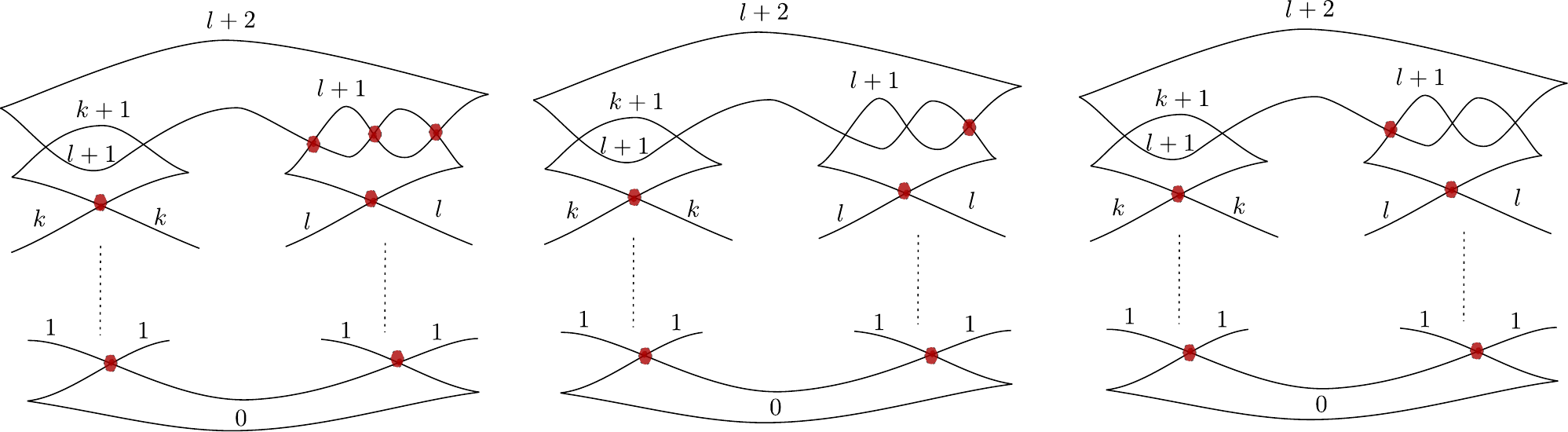}
        \includegraphics[scale=0.45]{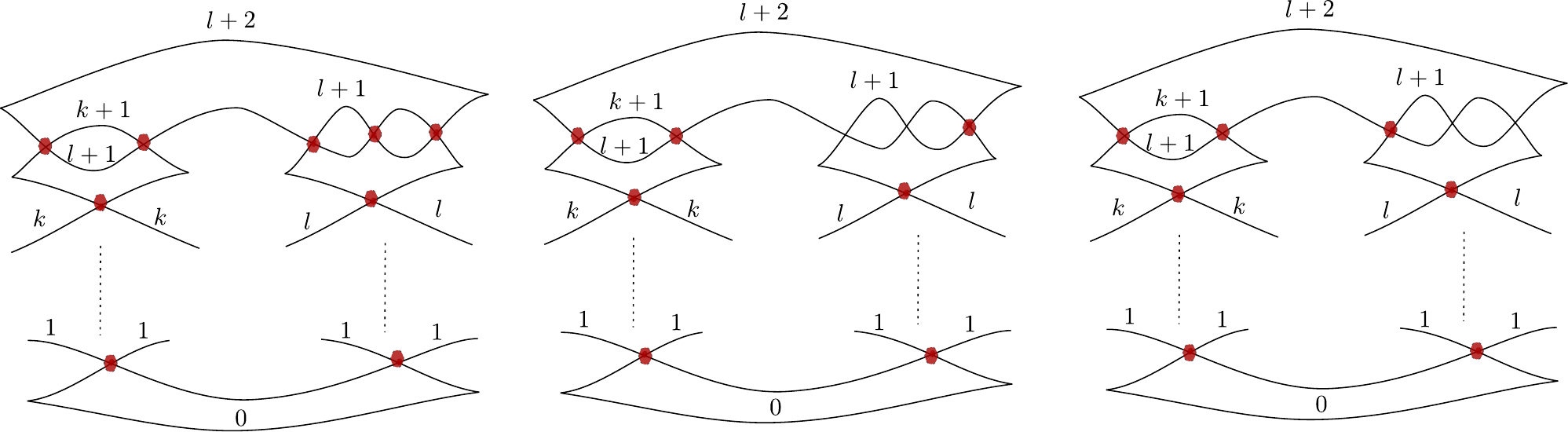}
        \caption{Above diagrams shows all possible $\rho$-graded rulings of $L[2,k:l,3]$, where the red dots indicate switches. The rulings shown in top row are $\rho$-graded for any $\rho$ since the switches happen at crossings with Maslov degree $0$. The rulings shown in the bottom row are $\rho$-graded whenever $\rho|k-l$.}
        \label{fig:ruling}
    \end{figure}
\end{example}
\begin{remark}
    The similar construction can be generalized to get maximal $tb$ representatives of $[a_1,\cdots,a_{2m+1}]$ whenever $\sum_{i=0}^ma_{2i+1}=\text{odd}$. 
\end{remark}
\begin{example}
    Legendrian realizations of positive cables of torus knots were classified in \cite{ELT}. For simplicity, we start by looking at $(r,s)$ cables of $T_{2,3}$. For $\frac{s}{r}<0,$ the maximum Thurston-Bennequin number is $rs$, and rotation numbers realized by these Legendrian knots are $$\{\pm (r+s(n+k))| k=(1+n),(1+n)-2,\cdots, -(1+n))\},$$ where $n$ is an integer that satisfies $-n-1 <\frac{r}{s}<-n$. Classification result shows that these knots are Legendrian simple. Notice, that the number of maximum Thurston-Bennequin Legendrian representatives increases as $n$ increases. This means that the difference between the extreme rotation numbers increases which increases the Cost between the Legendrian knots with extreme rotation numbers. Notice that $\frac{r}{s}$ can be made arbitrarily large. Hence given any number we can produce two Legendrian knots with the Cost between those two knots greater than the given number with the same Thurston-Bennequin number. A similar analysis can be done for cables of any positive torus knot. 
    \end{example}

\section{Cost Under the Connected Sum}
In \cite{EH}, the examples of Legendrian non-simple knot types are produced by taking the connected sum of negative torus knots. Lemma \ref{CostL11}, \ref{CostL12}, \ref{CostL14}, \ref{CostL15} and Proposition \ref{CostL16}, \ref{CostL17} are the results inspired by the techniques used in \cite{EH}.

\begin{Lemma}\label{CostL11}

    Let $K'$ and $K''$ be prime topological knot types with the property that every Legendrian representative of $K'$ destabilizes to a unique Legendrian knot $\Bar{K'}$ and every Legendrian representative of $K''$ destabilizes to a unique Legendrian knot $\Bar{K''}$. Then the knot type represented by $K'\#K''$ is Legendrian simple and each Legendrian representative of it destabilizes to $\Bar{K'}\#\Bar{K''}$.
\end{Lemma}

\begin{proof}
    Let $K_1,K_2$ be Legendrian knots in $\mathcal{L}(K'\#K'')$ with the same $tb$ and $rot$. Then there exist $K_1', K_2'\in \mathcal{L}(K')$ and $K_1'', K_2'' \in \mathcal{L}(K'')$ such that $K_1=K_1'\#K_1''$ and $K_2=K_2'\#K_2''$. There exist positive integers $p_1,p_2,q_1,q_2,n_1,n_2,m_1$ and $m_2$ such that $K_i'=(S^+)^{p_i}(S^-)^{n_i}(\bar{K'})$ and $K_i''=(S^+)^{q_i}(S^-)^{m_i}(\bar{K''})$ for $i=1,2.$ Then 
 \begin{align*}
        K_1&=(S^+)^{p_1}(S^-)^{n_1}(\bar{K'})\#(S^+)^{q_1}(S^-)^{m_1}(\bar{K''})\\
        &=\Bar{K'}\#(S^+)^{p_1+q_1}(S^-)^{n_1+m_1}(\Bar{K''}) \quad \text{ and }\\
        K_2&=(S^+)^{p_2}(S^-)^{n_2}(\bar{K'})\#(S^+)^{q_2}(S^-)^{m_2}(\bar{K''})\\
        &=\Bar{K'}\#(S^+)^{p_2+q_2}(S^-)^{n_2+m_2}(\Bar{K''}).&&
    \end{align*}

We get $p_1+q_1+n_1+m_1=p_2+q_2+m_2+n_2$ and $p_1+q_1-n_1-m_1=p_2+q_2-n_2-m_2$ since $tb(K_1)=tb(K_2)$ and $rot(K_1)=rot(K_2)$. Solving these two equations we get $p_1+q_1=p_2+q_2$ and $n_1+m_1=n_2+m_2$. Thus, $(S^+)^{p_1+q_1}(S^-)^{n_1+m_1}(\Bar{K''})=(S^+)^{p_2+q_2}(S^-)^{n_2+m_2}(\Bar{K''}).$ Hence, $K_1$ is Legendrian isotopic to $K_2$.

     For the proof of the second part, let $K\in \mathcal{L}(K'\#K'')$ be any Legendrian knot. Then $K=K_1'\#K_1''$ for some $K_1'\in \mathcal{L}(K')$ and $K_1''\in \mathcal{L}(K'')$. We know that $K_1'=(S^+)^{p}(S^-)^{n}(\Bar{K'})$ and $K_1''=(S^+)^{q}(S^-)^{m}(\Bar{K''})$ for some $p,n,q$ and $m$. Thus
     \begin{align*}
         K=(S^+)^{p}(S^-)^{n}(\Bar{K_1'})\#(S^+)^{q}(S^-)^{m}(\Bar{K_1''})=(S^+)^{p+q}(S^-)^{m+n}(\Bar{K'}\#\Bar{K''}).
     \end{align*}
\end{proof}

\begin{remark}\label{Re2}
 Lemma \ref{CostL11} can be extended to a connected sum of $n$ prime knots for any
natural number $n\geq 2$. That is, if $K_i, i = 1,\cdots, n$ are prime topological knot types with each having a unique non-destabilizable Legendrian representative $\Bar{K}_i$, then $K_1\#\cdots\#K_n$ is Legendrian simple and every Legendrian
representative of $K_1\#\cdots\#K_n$ is obtained by stabilizing $\Bar{K}_1\#\cdots \#\Bar{K}_n$. If the unique non-destabilizable hypothesis does not hold then the Legendrian simplicity is not guaranteed. For example, look at the connected sum of negative torus knots in \cite{EH}.
\end{remark}
\begin{corollary}
     The connected sum of two positive torus knots is Legendrian simple and every Legendrian representative destabilizes to a unique Legendrian knot with maximal Thurston-Bennequin number.
\end{corollary}
\begin{proof}
    Every positive torus knot is Legendrian simple and it has a unique maximal tb representative \cite{EH}. By Lemma \ref{CostL11}, the above statement follows. 
\end{proof}
The following result concerns the unique prime decomposition of a given Legendrian knot. In \cite{EH} a similar result for Legendrian knots having maximal tb is given. We generalize their result for the non-destabilizable Legendrian knots.
\begin{Lemma}\label{CostL12}
   If $K$ is a non-destabilizable Legendrian knot then it admits a unique prime decomposition up to possible permutations.
\end{Lemma}

\begin{proof}
    Let $K=K_1\#K_2\cdots \#K_n$ be a prime decomposition unique up to stabilization and possible permutations. If none of the $K_i$s destabilizes, then the prime decomposition is unique up to permutation using Remark \ref{newRemark}. If possible, let $K_i$ be obtained by stabilizing $\Bar{K_i}$, that is $K_i=(S^+)^p(S^-)^m(\Bar{K}_i)$ for some positive integer $p$ and $m$. Then 
    \begin{align*}
        K=&K_1\#\cdots (S^+)^p(S^-)^m(\Bar{K}_i) \cdots\# K_n\\
        =&(S^+)^p(S^-)^m(K_1\#\cdots \Bar{K}_i\cdots \#K_n).
    \end{align*}
    Thus we arrive at a contradiction.
\end{proof}

\begin{Lemma}\label{CostL14}
    Let $K=K_1\#\cdots\#K_n$ be a prime decomposition. If every Legendrian knot in $\mathcal{L}(K)$ destabilizes to a unique Legendrian knot $\Bar{K}$ then for each $i \in \{1,2,\cdots,n\}$, there exists a unique Legendrian representative $\Bar{K_i}$ of $K_i$ such that every Legendrian knot in $\mathcal{L}(K_i)$ destabilizes to $\Bar{K_i}$.
\end{Lemma}

\begin{proof}
    Let $\Bar{K}_i$ be a maximal $tb$ Legendrian knot in $\mathcal{L}(K_i) \ \forall \  1\leq i\leq n.$ If possible, let $\hat{K_i}\in \mathcal{L}
    (K_i)$ be another maximal $tb$ Legendrian knot. Then using Remark \ref{newRemark3}, $\Bar{K_1}\#\cdots \Bar{K_i}\#\cdots \Bar{K_n}$ and $
    \Bar{K_1}\#\cdots\hat{K_i}\#\cdots \Bar{K_n}$ are Legendrian knots in $\mathcal{L}(K)$, with maximal $tb$. Since $\Bar{K}$ is the unique maximal tb knot, $\Bar{K_1}\#\cdots \Bar{K_i}\#\cdots \Bar{K_n}$ and $
    \Bar{K_1}\#\cdots\hat{K_i}\#\cdots \Bar{K_n}$ are Legendrian isotopic to $\Bar{K}$. Thus, $\Bar{K_1}\#\cdots\hat{K_i}\#\cdots \Bar{K_n}
    =\Bar{K_1}\#\cdots \Bar{K_i}\#\cdots \Bar{K_n}$ which implies $\Bar{K}_i=\hat{K}_i$. Therefore there exists a unique maximal $tb$ representative of $K_i$ $\forall \ 1\leq i\leq n$.

    If possible, let $K_i'\in \mathcal{L}(K_i)$ be a non-maximal $tb$ and non-destabilizable Legendrian knot for some $i\in \{1,2,\cdots,n\}$. Now consider $\Bar{K}_1\#\cdots K_i'\cdots \#\Bar{K}_n\in \mathcal{L}(K)$. Since every Legendrian knot in $\mathcal{L}(K)$ destabilizes to $\Bar{K}$, for some $p,m\in \mathbb{Z}^+$ we get
    \begin{align*}
        \Bar{K}_1\#\cdots K_i'\cdots \#\Bar{K}_n&=(S^+)^{p}(S^-)^{m}(\Bar{K})\\
        &=\Bar{K_1}\#\cdots(S^+)^{p}(S^-)^m(\Bar{K}_i)\cdots\#\Bar{K_n}.
    \end{align*}
    Thus $K_i'=(S^+)^{p}(S^-)^{m}(\Bar{K}_i)$, which is a contradiction to the fact that $K_i'$ is non-destabilizable Legendrian knot. Hence, every Legendrian representative of $K_i$ destabilzes to $\Bar{K}_i \: \forall 1\leq i\leq n.$
\end{proof}

\begin{Lemma}\label{CostL15}
Let $K'$ and $K''$ be topological knot types with the property that every Legendrian representative of $K'$ destabilizes to a unique Legendrian knot $\Bar{K'}$ and every Legendrian representative of $K''$ destabilizes to a unique Legendrian knot $\Bar{K''}$. Then the knot type represented by $K'\#K''$ is Legendrian simple and every Legendrian representative of $K'\#K''$ destabilizes to $\Bar{K'}\#\Bar{K''}$.
\end{Lemma}
\begin{proof}
    The proof follows from Lemmas \ref{CostL11}, \ref{CostL14} and Remark \ref{Re2}.
\end{proof}

\begin{proposition}\label{CostL16}
Let $K'$ and $K''$ be prime knots of different topological knot types with maximal $tb$ Legendrian representatives $\Bar{K_1'},\Bar{K_2'}\in \mathcal{L}(K')$  and $\Bar{K_1''},\Bar{K_2''}\in \mathcal{L}(K'')$ such that $rot(\Bar{K_1'})+rot(\Bar{K_1''})=rot(\Bar{K_2'})+rot(\Bar{K_2''})$. Further assume that either $ \Bar{K_1'} $ is not Legendrian isotopic to $\Bar{K_2'}$ or $ \Bar{K_1''} $ is not Legendrian isotopic to $\Bar{K_2''}$. Then the knot type represented by $K'\#K''$ is Legendrian non-simple.
\end{proposition}
   
\begin{proof}
   
The Legendrian knots $\Bar{K_1'}\#\Bar{K_1''}$ and $\Bar{K_2'}\#\Bar{K_2''}$ are maximal $tb$ knots, thus their prime decomposition is unique up to possible permutations. Therefore, $\Bar{K_1'}\#\Bar{K_1''}$ is not Legendrian isotopic to $\Bar{K_2'}\#\Bar{K_2''}$. Hence, the knot type represented by $K'\#K''$ is Legendrian non-simple, since $tb(\Bar{K_1'}\#\Bar{K_1''})=tb(\Bar{K_2'}\#\Bar{K_2''})$ and $rot(\Bar{K_1'}\#\Bar{K_1''})=rot(\Bar{K_2'}\#\Bar{K_2''})$.
\end{proof}

\begin{proposition}\label{CostL17}
    Let $K$ be a topological prime knot type. If there exist maximal $tb$ Legendrian representatives $K_0,K_1,K_2$ and $K_3$ of $K$ with $rot(K_i)=rot(K_0)+i$ for $i=1,2,3$, then the knot type represented by $K\#K$ is a Legendrian non-simple.
\end{proposition}

\begin{proof}
    Consider Legendrian knots $K_0\#K_3$ and $K_1\#K_2$. Then $tb(K_0\#K_3)=tb(K_1\#K_2)$
 and $r(K_0\#K_3)=r(K_1\#K_2)$. Since $K_0\#K_3$ and $K_1\#K_2$ are maximal $tb$ knots, their prime decomposition is unique up to possible permutations. Hence, $K_0\#K_3$ is not Legendrian isotopic to $K_1\#K_2$.
 \end{proof}

\begin{Lemma}\label{CostL18}
    Let $K_1,L_1, K_2$ and $L_2$ be Legendrian knots of same topological knot type, then $$\Cost(K_1\#K_2,L_1\#L_2)\leq \min\{\Cost(K_1,L_1)+\Cost(K_2,L_2),\Cost(K_1,L_2)+\Cost(K_2,L_1)\}.$$
\end{Lemma}
\begin{proof}
 Let $\Cost(K_i,L_j)=p_{ij}+n_{ij}+q_{ij}+m_{ij}$ with $$(S^+)^{p_{ij}}(S^-)^{n_{ij}}(K_i)=(S^+)^{q_{ij}}(S^-)^{m_{ij}}(L_j), \: i,j=1,2.$$ Then, 

    \begin{align}
        \nonumber (S^+)^{p_{11}+p_{22}}(S^-)^{n_{11}+n_{22}}(K_1\#K_2)&=(S^+)^{p_{11}}(S^-)^{n_{11}}(K_1)\#(S^+)^{p_{22}}(S^-)^{n_{22}}(K_2)\\
       \nonumber &=(S^+)^{q_{11}+q_{22}}(S^-)^{m_{11}+m_{22}}(L_1\#L_2).
        \end{align}
        Thus,
        \begin{align}
        \nonumber \Cost(K_1\#K_2,L_1\#L_2)&\leq \sum_{i=1}^2p_{ii}+n_{ii}+q_{ii}+m_{ii}\\
        &=\Cost(K_1,L_1)+\Cost(K_2,L_2).\label{eq:Costeq1}
    \end{align}
    Now,
     \begin{align}
       \nonumber (S^+)^{p_{12}+p_{21}}(S^-)^{n_{12}+n_{21}}(K_1\#K_2)&=(S^+)^{p_{12}}(S^-)^{n_{12}}(K_1)\#(S^+)^{p_{21}}(S^-)^{n_{21}}(K_2)\\
       \nonumber &=(S^+)^{q_{12}}(S^-)^{m_{12}}(L_2)\#(S^+)^{q_{21}}(S^-)^{m_{21}}(L_1)\\
       \nonumber &=(S^+)^{q_{12}+q_{21}}(S^-)^{m_{21}+m_{21}}(L_2\#L_1).
       \end{align}
       Thus,
       \begin{align}
        \Cost(K_1\#K_2,L_1\#L_2)&\leq \Cost(K_1,L_2)+\Cost(K_2,L_1).\label{eq:Costeq2}
    \end{align}

   Using equations \ref{eq:Costeq1} and \ref{eq:Costeq2}, we have 
   \begin{align*}
     \Cost(K_1\#K_2,L_1\#L_2)  \leq \min\{\Cost(K_1,L_1)+\Cost(K_2,L_2),\Cost(K_1,L_2)+\Cost(K_2,L_1)\}.
   \end{align*}
   
    \end{proof}

\begin{Lemma}\label{CostL19}
    Let $K_1,L_1$ and $K_2,L_2$ be Legendrian knots of two different topological knot types, then $\Cost(K_1\#K_2,L_1\#L_2)\leq \Cost(K_1,L_1)+\Cost(K_2,L_2)$.
\end{Lemma}
\begin{proof}
   The proof follows from the first part of the proof of Lemma \ref{CostL18}.
\end{proof}

\begin{Lemma}\label{CostL20}
    Let $K_1,L_1$ and $K_2,L_2$ be Legendrian knots of two different topological knot types such that $K_1, K_2$ destabilizes to $ L_1,L_2$, respectively. Then $\Cost(K_1\#K_2,L_1\#L_2)=\Cost(K_1,L_1)+\Cost(K_2,L_2)$.  
\end{Lemma}
\begin{proof}
As $K_1, K_2$ destabilizes to $ L_1,L_2$, respectively, there exist non-negative integers $p,n,q$ and $m$ such that the following holds
 \begin{itemize}
     \item $K_1=(S^+)^p(S^-)^n(L_1),$ and $ K_2=(S^+)^{q}(S^-)^m(L_2)$
     \item $\Cost (K_1, L_1) = p+n$ and $\Cost(K_2, L_2)=q+m$ by Remark \ref{NewRemark2}
 \end{itemize} 
 Then the lemma follows from Remark \ref{NewRemark2} and the observation 
 \begin{align*}
     K_1\#K_2&=(S^+)^p(S^-)^n(L_1)\#(S^+)^{q}(S^-)^m(L_2)\\
     &=(S^+)^{p+q}(S^-)^{n+m}(L_1\#L_2).
 \end{align*}

\end{proof}

\begin{Lemma}\label{CostL21}
Let $K_1,L_1$ and $K_2,L_2$ be Legendrian knots of two different topological knot types such that $K_1=(S^+)^p(S^-)^n(L_1)$ and $ L_2=(S^+)^q(S^-)^{m}(K_2)$. Then $$\Cost(K_1\#K_2,L_1\#L_2)=|p-q|+|n-m|.$$
\end{Lemma}
\begin{proof}
We will prove this Lemma in cases. 
    
    \noindent \textbf{Case I:} $p\geq q,n\geq m$.
    
  \noindent We prove $K_1\#K_2$ destabilizes to $L_1\#L_2$ as follows:
   \begin{align*}
       K_1\#K_2&=(S^+)^p(S^-)^n(L_1\#K_2)\\
       &=(S^+)^{p-q}(S^-)^{n-m}(L_1\#(S^+)^q(S^-)^{m}(K_2))\\
       &=(S^+)^{p-q}(S^-)^{n-m}(L_1\#L_2).&&
   \end{align*} 
   Using Remark \ref{NewRemark2}, we have
   \begin{align*}
       \Cost(K_1\#K_2,L_1\#L_2)&=p-q+n-m\\
       &=|\Cost(K_1,L_1)-\Cost(K_2,L_2)|.
   \end{align*}
   
 \noindent \textbf{Case II:} $p\leq q,n\leq m$.

   \noindent Using the similar technique as used in Case I, we prove that $K_1\#K_2$ destabilizes to $L_1\#L_2$.
   \begin{align*}
       L_1\#L_2&=(S^+)^q(S^-)^{m}(L_1\#K_2)\\
       &=(S^+)^{q-p}(S^-)^{m-n}((S^+)^p(S^-)^n(L_1)\#K_2)\\
       &=(S^+)^{q-p}(S^-)^{m-n}(K_1\#K_2).&&
   \end{align*} 
   Therefore, by Remark \ref{NewRemark2}
   \begin{align*}
       \Cost(K_1\#K_2,L_1\#L_2)&=q-p+m-n\\
       &=|\Cost(K_1,L_1)-\Cost(K_2,L_2)|.
   \end{align*}

   \noindent \textbf{Case III:} $p\geq q,n\leq m$.

\noindent Since $ rot(K_1)=rot(L_1)-n+p$, and $,rot(L_2)=rot(K_2)-m+q$, using Lemma \ref{CostL8}, we get
\begin{align*}
    \Cost(K_1\#K_2,L_1\#L_2)&\geq |rot(K_1\#K_2)-rot(L_1\#L_2)|\\
    &=p-q+m-n.
\end{align*}
The reverse inequality follows from the computations done below. 

   \begin{align*}
       (S^-)^{m-n}(K_1\#K_2)&=(S^-)^{m-n}((S^+)^p(S^-)^n(L_1)\#K_2)\\
       &=(S^+)^p(S^-)^{m}(L_1\#K_2)\\
       &=(S^+)^{p-q}(L_1\#(S^+)^q(S^-)^m(K_2))\\
       &=(S^+)^{p-q}(L_1\#L_2).
   \end{align*}

   \noindent \textbf{Case IV:} $p\leq q,m\leq n$.

      Using Lemma \ref{CostL8}, we have
\begin{align*}
    \Cost(K_1\#K_2,L_1\#L_2)&\geq |rot(K_1\#K_2)-rot(L_1\#L_2)|\\
    &=|p-q+m-n|=q-p+n-m.
\end{align*}

   The reverse inequality follows from the following computations:
\begin{align*}
       (S^+)^{q-p}(K_1\#K_2)&=(S^+)^{q-p}((S^+)^p(S^-)^n(L_1)\#K_2)\\
       &=(S^+)^q(S^-)^{n}(L_1\#K_2)\\
       &=(S^-)^{n-m}(L_1\#(S^+)^q(S^-)^m(K_2))\\
       &=(S^-)^{n-m}(L_1\#L_2).
   \end{align*}

   Hence, we conclude from the above cases that $$\Cost(K_1\#K_2,L_1\#L_2)=|p-q|+|n-m|.$$
\end{proof}

\begin{Lemma}\label{CostL22}
    Let $\mathcal{K}_1$ and $\mathcal{K}_2$ be different topological knot types which are Legendrian simple. Let $K_1,L_1\in \mathcal{K}_1$ and $K_2,L_2\in\mathcal{K}_2$ be Legendrian knots with maximal $tb$, then 
    
    \noindent (i) If $rot(K_i)\leq rot(L_i),i=1,2$, then $$\Cost(K_1\#K_2,L_1\#L_2)=\Cost(K_1,L_1)+\Cost(K_2,L_2).$$
       
\noindent (ii) If $rot(K_1)\leq rot(L_1)$ and $rot(K_2)\geq rot(L_2)$, then
        \begin{align*}
            |\Cost(K_1,L_1)-\Cost(K_2,L_2)|&\leq \Cost(K_1\#K_2,L_1\#L_2)\\
            &\leq \max\{Cost(K_1,L_1),Cost(K_2,L_2)\}.&&
        \end{align*}
    
\end{Lemma}

\begin{proof}

By Lemma \ref{CostL9}, we have $(S^+)^{r_i}(K_i)=(S^-)^{r_i}(L_i)$, where $$2r_i=|rot(K_i)-rot(L_i)|=\Cost(K_i,L_i),\:i=1,2.$$  

 \noindent(i)  $rot(K_i)\leq rot(L_i),i=1,2$.
    
Using Lemma \ref{CostL19}, we have
\begin{align*}
    \Cost(K_1\#K_2,L_1\#L_2)&\leq \Cost(K_1,L_1)+\Cost(K_2,L_2)\\
    &=2(r_1+r_2).
\end{align*}

The reverse inequality can be derived using Lemma \ref{CostL8} in the following way.
\begin{align*}
    \Cost(K_1\#K_2,L_1\#L_2)&\geq |rot(K_1\#K_2)-rot(L_1\#L_2)|\\
    &=(rot(L_1)-rot(K_1))+(rot(L_2)-rot(K_2))\\
    &=2(r_1+r_2). 
\end{align*}
Hence, $ \Cost(K_1\#K_2,L_1\#L_2)=\Cost(K_1,L_1)+\Cost(K_2,L_2)$.

   \noindent (ii) $rot(K_1)\leq rot(L_1)$ and $rot(K_2)\geq rot(L_2)$.

    Observe that,
    \begin{align*}
        rot(K_1\#K_2)&=rot(K_1)+rot(K_2)\\
        &=rot(K_1)+rot(L_2)+2r_2,\: \text{ and }\\
        rot(L_1\#L_2)&=rot(L_1)+rot(L_2)\\
       & =rot(K_1)+rot(L_2)+2r_1.
    \end{align*}
    Then, $|rot(K_1\#K_2)-rot(L_1\#L_2)|=2r$, where $r=|r_1-r_2|$. Using Lemma \ref{CostL8}, we have
   $$\Cost(K_1\#K_2,L_1\#L_2)\geq 2r=|\Cost(K_1,L_1)-\Cost(K_2,L_2)|.$$ 
   Now, we consider two cases to prove 
     $$\Cost(K_1\#K_2,L_1\#L_2)\leq \max\{\Cost(K_1,L_1),\Cost(K_2,L_2)\}.$$ 
     First consider $r_2\leq r_1$. Then
 \begin{align*}
       (S^+)^{r_1}(K_1\#K_2)&=(S^-)^{r_1}(L_1\#K_2)\\
       &=(S^-)^{r}(L_1)\#(S^-)^{r_2}(K_2)\\
       &=(S^-)^{r}(L_1)\#(S^+)^{r_2}(L_2)\\
       &=(S^-)^{r}(S^+)^{r_2}(L_1\#L_2).
   \end{align*}
       
     Thus, $\Cost(K_1\#K_2,L_1\#L_2)\leq r+r_1+r_2=2r_1.$

    Now consider the case with $r_2\geq r_1$. Then
    \begin{align*}
        (S^-)^{r_2}(K_1\#K_2)&=(S^+)^{r_2}(K_1\#L_2)\\
        &=(S^+)^{r_1}(K_1)\#(S^+)^{r}(L_2)\\
        &=(S^-)^{r_1}(L_1)\#(S^+)^{r}(L_2)\\
        &=(S^-)^{r_1}(S^+)^{r}(L_1\#L_2).
    \end{align*}
     Thus, $\Cost(K_1\#K_2,L_1\#L_2)\leq r+r_1+r_2=2r_2.$ Hence,
     $$\Cost(K_1\#K_2,L_1\#L_2)\leq \max\{\Cost(K_1,L_1),\Cost(K_2,L_2)\}.$$
     
   \end{proof}
An application of the results proved in this article is demonstrated in the following example.
   \begin{example}
       Let $K_1$ and $K_2$ be oriented left-handed and right-handed Legendrian trefoils as shown in Figure \ref{fig:CostOLH}. Let $\Bar{K_1}$ be the Legendrian knot obtained by reversing the orientation of $K_1$. Note that $tb(K_1)=-6$ and $tb(K_2)=1$ and $rot(K_1)=-1, rot(K_2)=0$. Then $\Cost(K_1,\Bar{K_1})=|rot(K_1)-rot(\Bar{K_1})|=2$, since left-handed trefoil is a Legendrian simple knot type. The Legendrian knots $K_1\#K_2$ and $\Bar{K_1}\#K_2$ are of the same topological knot type since left-handed trefoil knot is invertible. Then by Lemma \ref{CostL22}
       $ \Cost(K_1\#K_2,\Bar{K_1}\#K_2)= 2$.
       \begin{figure}
           \centering
           \includegraphics[scale=0.5]{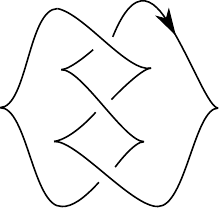}
           \hspace{30pt}
           \includegraphics[scale=0.6]{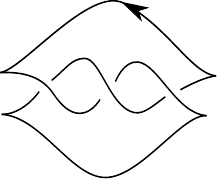}
           \caption{Oriented left-handed and right-handed Legendrian trefoils.}
           \label{fig:CostOLH}
       \end{figure}
   \end{example}

\section{Conclusion}\label{Costgraph}

For a topological knot type $\mathcal{K}$, let $\mathcal{L(K)}/\sim$ denote the set of equivalence classes of Legendrian representatives of $\mathcal{K}$ up to Legendrian isotopy. The Cost function gives a metric on the set $\mathcal{L(K)}/
\sim$ in the following way. The map $$d_C: (\mathcal{L(K)}/\sim)\times (\mathcal{L(K)}/\sim)\rightarrow \Z^+\cup\{0\}$$ defined as $d_C([K],
[\Tilde{K}])=\Cost(K,\Tilde{K})$ for all $[K],[\Tilde{K}]\in \mathcal{L(K)}/\sim$ is a metric. The map $d_C$ satisfies the first two properties of a metric (see Theorem \ref{CostL4} and Remark \ref{re:Cost1}). For the triangle 
inequality, consider $[K_1],[K_2],[K_3]\in \mathcal{L(K)}/\sim$. Let $F_i$ be a front projection of $K_i,\: i=1,2,3$. Let $SF_{1}\in \mathcal{SF}(F_1,F_2)$ and $SF_{2}\in \mathcal{SF}(F_2,F_3)$ be such that $
\Cost(K_1,K_2)=\Cost(F_1,F_2,SF_{1})$ and $\Cost(K_2,K_3)=\Cost(F_2,F_3,SF_{2})$. Therefore the sequence $(SF_{1},SF_{2})\in \mathcal{SF}(F_1,F_3)$. Thus \begin{align*}
    \Cost(K_1,K_3)&\leq \Cost(F_1,F_3,(SF_{1},SF_{2}))\\
    &=\Cost(F_1,F_2,SF_1)+\Cost(F_2,F_3,SF_2)\\
    &=\Cost(K_1,K_2)+\Cost(K_2,K_3).
\end{align*}
Hence, $d_C$ is a metric.

Using the Cost function we define a graph invariant of the topological knot. We associate a graph $G_{\mathcal{K}}$ to a topological knot type $\mathcal{K}$ in the following way. Fix the set of vertices for $G_{\mathcal{K}}$ to be $\mathcal{L(K)}/\sim$. Two vertices $[K],[\Tilde{K}]\in \mathcal{L(K)}/\sim$ are joined by an edge if and only if $\Cost(K,\Tilde{K})=1$. We refer to this graph as the graph of \emph{Legendrian representatives} of ${\mathcal{K}}$. The Cost function graph for unknot is shown in Figure \ref{fig:Costgraph}. 

\begin{figure}
    \centering
    \includegraphics[scale=0.7]{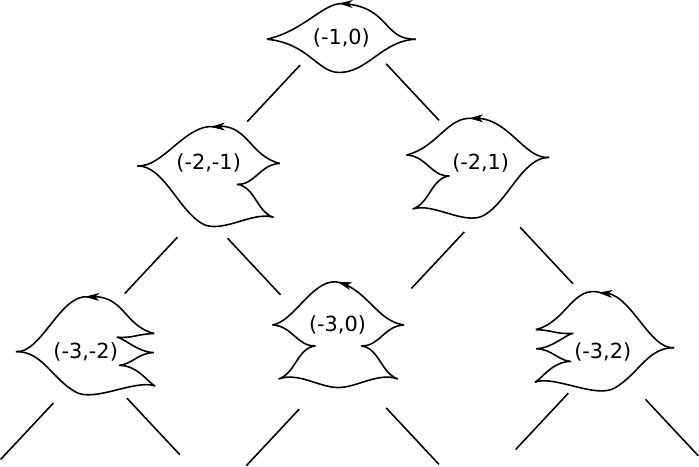}
    \caption{The Cost function graph for unknot.}
    \label{fig:Costgraph}
\end{figure}
\begin{remark}
     For Legendrian simple knot types, the Cost function graph is the same as the mountain range defined in \cite{E}.
\end{remark}

We now raise a few questions, given below, that the authors are currently investigating.
\begin{enumerate}
    \item In the class of prime knots, does the graph $G_{\mathcal{K}}$ together with the value of maximal $tb$ detect unknot? More explicitly, for a prime knot $K$, if $G_{\mathcal{K}}$ is isomorphic to that of unknot as a graph and maximal $tb$ of $\mathcal{K}$ equals $-1$ then does it necessarily follow that $\mathcal{K}$ is unknot?
    \item Is there a relation between the values of the Cost function for a given knot type and Eliashberg-Chekanov DGAs of corresponding Legendrian representatives?
    \item From a computational viewpoint, is there an algorithm or a procedure to generate the graph $G_{\mathcal{K}}$ for a given knot type $\mathcal{K}$?
\end{enumerate}

\bibliographystyle{plain}
\bibliography{references.bib}
\vspace{10pt}
\end{document}